\documentclass{amsart}

\usepackage{ae}
\usepackage[all]{xy}
\usepackage{graphicx,amsfonts,amssymb,amsmath}
\usepackage{natbib}
\usepackage{enumerate}

\usepackage[latin1]{inputenc}
\usepackage[T1]{fontenc}
\usepackage[english]{babel}
\usepackage[cyr]{aeguill}

\newcommand{\chapeau}{{\rlap{\smash{\hbox{\lower4pt\hbox{\hskip1pt$\widehat{\phantom{u}}$}}}}}\mbox{ }}
\usepackage[OT2,T1]{fontenc}
\DeclareSymbolFont{cyrletters}{OT2}{wncyr}{m}{n}
\DeclareMathSymbol{\sha}{\mathalpha}{cyrletters}{"58}

 \newtheorem{thm}{Theorem}[section]
 \newtheorem*{thm*}{Theorem}

 \newtheorem{thm'}[thm]{Theorem'}

 \newtheorem*{pb*}{\textit{Open question}}
 \newtheorem*{cor*}{\textit{Corollary}}
 
 \newtheorem{conj}[thm]{Conjecture}
 \newtheorem{cor}[thm]{Corollary}
 \newtheorem{lem}[thm]{Lemma}
 \newtheorem{prop}[thm]{Proposition}
 \theoremstyle{definition}
 \newtheorem{defn}[thm]{Definition}
 \theoremstyle{definition}
 
 \theoremstyle{remark}
 
 \theoremstyle{remark}
 \newtheorem{rem}[thm]{Remark}
 \theoremstyle{remark}

 \numberwithin{equation}{subsection}

 \newcommand{\To}{\longrightarrow}

 \newcommand{\BM}{Brauer\textendash Manin}
 \newcommand{\BMo}{Brauer\textendash Manin obstruction}
 \newcommand{\E}{$(\textup{E})$~}
 \newcommand{\ZZ}{\textup{Z}_0}
 \newcommand{\Hil}{\textsf{Hil}}
 \newcommand{\Sym}{\textup{Sym}}
 
 \renewcommand{\div}{\textup{div}}

 \renewcommand{\P}{\mathbb{P}}
 \newcommand{\Q}{\mathbb{Q}}
 \newcommand{\Z}{\mathbb{Z}}

 \newcommand{\Br}{\textup{Br}}
 
 \newcommand{\CH}{\textup{CH}}
  \newcommand{\eff}{\textup{eff}}
 \renewcommand{\H}{\textup{H}}
 \newcommand{\Hom}{\textup{Hom}}
 \newcommand{\Spec}{\textup{Spec}}
 \renewcommand{\dim}{\textup{dim}}

\usepackage{hyperref}
\hypersetup{colorlinks,linkcolor=blue,citecolor=blue,}

\begin{document}

\title[zero-cycles on products]
{Compatibility of weak approximation for zero-cycles on products of varieties}

\author{ Yongqi LIANG  }

\address{Yongqi LIANG
\newline 96 Jinzhai Road,
\newline CAS Wu Wen-Tsun Key Laboratory of Mathematics,
\newline School of Mathematical Sciences,
\newline University of Science and Technology of China,
\newline Hefei, Anhui, 230026 P. R. China}

\email{yqliang@ustc.edu.cn}

\thanks{\textit{Key words} : 0-cycle, Brauer\textendash Manin obstruction, weak approximation}

\thanks{\textit{MSC 2020} : 14G12   (11G35, 14C25)}

\date{\today.}



\maketitle

\begin{abstract}
Zero-cycles are conjectured to satisfy weak approximation with \BMo~for proper smooth varieties defined over number fields. Roughly speaking, we prove that the conjecture is compatible for products of rationally connected varieties, K3 surfaces, Kummer varieties, and one curve.
\end{abstract}


\tableofcontents

\section{Introduction}

We consider proper smooth geometrically integral varieties $V$ defined over a number field $k$. The weak approximation property for rational points on $V$ describes the relation between rational points of $V$ over $k$ and rational points of $V$ over all its completions $k_v$. By using the Brauer group $\Br(V)$, Yu. I. Manin \cite{Manin} introduced a pairing to study weak approximation for rational points. In \cite{CT95}, J.-L. Colliot-Th\'el\`ene extended the \BM~pairing to study 0-cycles, see also \S \ref{BMpairing}. From local to global, we consider weak approximation with \BMo~for 0-cycles on $V$, which is very closely related to the exactness of the following sequence induced by the \BM~pairing
$$\varprojlim_n\CH_0(V)/n\To\prod_{v\in\Omega_k}\varprojlim_n\CH_0'(V_{k_v})/n\To\Hom(\Br(X),\Q/\Z)\leqno(\textup{E})$$
where the inverse limit is taken over all positive integers on the cokernel of multiplication by $n$ of modified Chow groups, see \S \ref{WABMnotion} for definition and more details. The exactness of \E means roughly that 0-cycles of any degree satisfy weak approximation with \BMo. The exactness is conjectured by K. Kato and S. Saito \cite[\S 7]{KatoSaito86} and by J.-L. Colliot-Th\'el\`ene \cite[\S 1]{CT95} for all proper smooth geometrically integral varieties defined over number fields. When $V$ is a curve, the conjecture is implied by  the finiteness of the Tate\textendash Shafarevich group of its jacobian. This was proved by S. Saito \cite[(7-1) and (7-5)]{Saito89} and J.-L. Colliot-Th\'el\`ene \cite[\S 3]{CT99HP0-cyc}, see also \cite[Remarque 1.1(iv)]{Wittenberg}. For varieties with a fibration structure, the conjecture was also proved under divers assumptions on the fibration, we refer to the survey by O. Wittenberg \cite{WittSLC18} for more information. Recently, Y. Harpaz and O. Wittenberg \cite{HarWittJAMS} proved the conjecture for smooth compactifications of homogeneous spaces of connected linear algebraic groups.

In this paper, we consider the compatibility of the conjecture for products of varieties. We prove the following theorem, which provides evidence of the conjecture.

\begin{thm*}
Let $X$ and $Y$ be geometrically rationally connected varieties defined over a number field $k$. Suppose that \E is exact for $X$ and $Y$ after any finite extension of the ground field. Then \E is also exact for $X\times Y$.
\end{thm*}
We also prove variants of this result in the following respects, where the statements are similar but not exactly the same.
\begin{itemize}
\item[-] the number of factors of the product is allowed to be arbitrary;
\item[-] the factors are allowed to be K3 surfaces;
\item[-] the factors are allowed to be Kummer varieties;
\item[-] at most one factor of the product is allowed to be a curve of positive genus.
\end{itemize}
For precise statements, please refer to Theorem \ref{mainthm} and its consequences in \S \ref{mainresult}.
Before this result, some cases for the compatibility of the conjecture for products of varieties are already known:
\begin{itemize}
\item[-] If $Y$ is the projective space $\P^n$, then the exactness of \E for $X$ implies its exactness for $X\times\P^n$. This is well known, see Lemma \ref{homotopylemma}.
\item[-] As a very particular case of the main result of \cite{Wittenberg}, O. Wittenberg proved the case where $X$ is a geometrically rationally connected variety and $Y$ is a curve whose jacobian has finite Tate\textendash Shafarevich group.
\item[-] In the previous work of the author \cite{Liang10}, we proved the case where $X$ is a geometrically rationally connected and $Y$ is one of the following varieties:
    \begin{enumerate}
    \item a Ch\^atelet surface;
    \item a smooth compactification of a homogeneous space of a connected linear algebraic group with connected stabilizer;
    \item a smooth compactification of a homogeneous space of a semi-simple simply connected algebraic group with abelian stabilizer;
    \item  more generally, $Y$ is geometrically rationally connected and satisfies weak approximation with \BMo~for \emph{rational points} after any finite extension of the ground field.
    \end{enumerate}
\end{itemize}
The fourth case was not satisfactory in the sense that we  made an assumption on the arithmetic of rational points, which implies the exactness of \E according to \cite[Theorem A]{Liang4}. One could discuss only the arithmetic of 0-cycles  without being aware of any information on rational points. For example, the result of this paper can be applied to the case where $Y$ is a smooth compactification of a homogeneous space of a connected linear algebraic group with \emph{arbitrary} stabilizer, for which \E is exact over any number field by \cite[Th\'eor\`em A]{HarWittJAMS}. On such a variety $Y$, weak approximation with \BMo~for rational points is still unknown.

Most ingredients of the proof are known in the literature except the simple but crucial observation: Lemma \ref{keylemma}, which allows us to discuss the weak approximation property for 0-cycles on a product of varieties. We subtly combines them to cook the proof. The paper is organized as follows. We recall terminology and basic notions about weak approximation for 0-cycles in \S \ref{notions} and state the main result and its consequences in \S \ref{mainresult}. The detailed proof is given in \S \ref{proofsection} after some preliminaries.

\section{Terminology and basic notions}\label{notions}

\subsection{Notation}
In this paper, the ground field $k$ is always of characteristic $0$. In most statements, it is a number field. We fix an algebraic closure $\bar{k}$ of $k$ and denote by $\Gamma_k$ the absolute Galois group of $k$. We denote by $\Omega_k$ the set of places of $k$. For a place $v\in \Omega_k$, we denote by $k_v$ the completion of $k$ with respect to $v$, it is either a finite extension of $\mathbb{Q}_p$ or one of the archimedean local fields $\mathbb{R}$ or $\mathbb{C}$. $S$ will be a finite subset of $\Omega_k$, on which we will discuss weak approximation properties for 0-cycles.

The word ``variety" means a separated scheme of finite type defined over a field $k$. If $K$ is a field extension of $k$, for a $k$-variety $V$ we write $V_K$ for $V\times_{\Spec(k)}\Spec(K)$. When $X$ and $Y$ are varieties defined over a same field, the product $X\times Y$ means the fiber product over their field of definition.

The Brauer group $\Br(V)$ of a $k$-variety is the second \'etale cohomology group $\H^2_{\scriptsize{\textup{\'et}}}(V,\mathbb{G}_{\textup{m}})$. When $V=\Spec(k)$ is a spectrum of a field, we write $\Br(k)$ instead of $\Br(\Spec(k))$.  The image of the natural homomorphism $\Br(k)\To\Br(V)$ is denoted by $\Br_0(V)$. The kernel of the base change homomorphism $\Br(V)\to \Br(V_{\bar{k}})$ is denoted by $\Br_1(V)$.

When $A$ is an abelian group and $n$ is a positive integer, we denote by $A/n$ the cokernel of the multiplication by $n$ homomorphism. We denote by $A\{n\}$ the $n$-primary part of $A$ which is defined to be the union for $r\in\mathbb{N}$ of the $n^r$-torsion subgroup of $A$.

\subsection{Brauer\textendash Manin pairing for 0-cycles}\label{BMpairing}
The set of 0-cycles on a variety $V$, denoted by $\ZZ(V)$, is the free abelian group generated by its closed points. Rational equivalence $\sim$ is a equivalent relation defined on $\ZZ(V)$, the quotient is the Chow group of 0-cycles $\CH_0(V)$.
A 0-cycle is written as a $\Z$-linear combination of closed points of $V$. If the coefficients are all positive, the 0-cycle is called \emph{effective}. The set of effective 0-cycles of degree $\Delta$ on $V$ can be identified with the set of rational points of the $k$-variety $\Sym^\Delta_V$ \textemdash~the symmetric product of $V$ over $k$.
An effective 0-cycle is called \emph{separable} if it is a sum of distinct closed points.

When $V$ is a proper smooth geometrically integral variety defined over a number field $k$, Yu. I. Manin defined a pairing between the Brauer group $\Br(V)$ and the product of local rational points $\prod_{v\in\Omega_k}V(k_v)$, cf. \cite{Manin}. In \cite{CT95}, J.-L. Colliot-Th\'el\`ene extended the Brauer\textendash Manin pairing to 0-cycles:
$$\Br(V)\times \prod_{v\in\Omega_k}\ZZ(V_{k_v})\To \Q/\Z$$
$$(b,(z_v)_{v\in\Omega_k})\mapsto \sum_{v\in\Omega_k}\textup{inv}_v(b(z_v))$$
where the $\textup{inv}_v:\Br(k_v)\to\Q/\Z$ is the local invariant at $v$ given by local class field theory and where the evaluation of $b\in\Br(V)$ at a 0-cycle $z_v$ is defined as follows. If $z_v=\sum_Pn_PP$ then $b(z_v)=\sum_Pn_P\textup{cores}_{k_v(P)|k_v}(b(P))$ where $b(P)\in\Br(k_v(P))$ is the pull-back of $b\in\Br(V)$ by the rational point $P$ viewed as a morphism $\Spec(k_v(P))\to V$ of schemes and where $\textup{cores}_{k_v(P)|k_v}:\Br(k_v(P))\To\Br(k_v)$ is the correstriction homomorphism.
This pairing factorises through Chow groups of 0-cycles.
Moreover, it factorises also through the modified Chow groups $\CH'_0(V_{k_v})$ defined as follows. When $v$ is a non archimedean local field, the modified Chow group is the usual Chow group; when $v$ is a complex place, it is $0$; when $v$ is a real place then $\CH'_0(V_{\mathbb{R}})=\CH_0(V_{\mathbb{R}})/\textup{N}_{\mathbb{C}|\mathbb{R}}\CH_0(V_\mathbb{C})$.
By global class field theory, we have an exact sequence $$0\to\Br(k)\to\bigoplus_{v\in\Omega_k}\Br(k_v)\to\Q/\Z\to0$$ which implies that global 0-cycles are annihilated by the whole Brauer group under the pairing.
As $V$ is a smooth variety, its Brauer group is a torsion group, we deduce a complex $$\varprojlim_n\CH_0(V)/n\To\prod_{v\in\Omega_k}\varprojlim_n\CH_0'(V_{k_v})/n\To\Hom(\Br(V),\Q/\Z)\leqno(\textup{E})$$
where the projective limit is taken over $n\in\mathbb{N}$.  The following conjecture on the exactness of \E means roughly that local-global principle and weak approximation for 0-cycles are completely controlled by the Brauer group. We will explain more in the next subsection.
\begin{conj}[Kato\textendash Saito and  Colliot-Th\'el\`ene]
\ \\ \indent The sequence \E is exact for all proper smooth varieties.
\end{conj}

We remark that the subgroup $\Br_0(V)$ gives no contribution to the pairing with a family of local 0-cycles \emph{of the same degree}. This will be the case in our proofs and we will require the finiteness of $\Br(V)/\Br_0(V)$ as an assumption.

\subsection{Weak approximation for 0-cycles}\label{WABMnotion}
Let $V$ be a proper smooth geometrically integral variety defined over a number field $k$. Let $\delta$ be an integer.
\begin{defn}
\begin{enumerate}\item[]
\item We say that $V$ satisfies \emph{weak approximation for 0-cycles of degree $\delta$} if, for any positive integer $n$ and for any finite set $S$ of places of $k$, given a family of local 0-cycles $(z_v)_{v\in\Omega_k}$ of degree $\delta$ on $V$ then there exists a global 0-cycle $z=z_{n,S}$ of degree $\delta$ such that $z$ and $z_v$ has the same image in $\CH_0(V_{k_v})/n$ for all $v\in S$.
\item\label{def2} We say that $V$ satisfies \emph{weak approximation with \BMo~for 0-cycles of degree $\delta$} if the same conclusion holds for all those families $(z_v)_{v\in\Omega_k}$ of 0-cycles of degree $\delta$ orthogonal to the Brauer group $\Br(V)/\Br_0(V)$.
\item Let $c$ be a positive integer, we also say that $V$ satisfies \emph{weak approximation with $c$-primary \BMo~for 0-cycles of degree $\delta$} if only the $c$-primary part of $\Br(V)/\Br_0(V)$ is considered in the previous definition (\ref{def2}).
\end{enumerate}
\end{defn}

There are very close relations between the notion of weak approximation for 0-cycles and the exactness of \E for $V$. We summarise several known results for the convenience of the reader.
\begin{itemize}
\item[-] Suppose that  \E is exact for $V$, if there exists a family of local 0-cycles of degree $1$ orthogonal to $\Br(V)$ then there exists a global 0-cycle of degree $1$ on $V$, cf. \cite[Remarque 1.1(iii)]{Wittenberg}.
\item[-] Assuming the existence of a global 0-cycle of degree $1$, the exactness of \E for $V$ implies that $V$ satisfies weak approximation with \BMo~for 0-cycles of degree $\delta$ for any integer $\delta$, cf. \cite[Proposition 2.2.1]{Liang4}.
\item[-] The exactness of \E for $V$ implies that $V$ satisfies weak approximation with \BMo~for 0-cycles of degree $1$. This is a consequence of the previous two statements.
\item[-] Suppose that $V$ is a geometrically rationally connected variety. If $V_K$ satisfies weak approximation with \BMo~for 0-cycles of degree $1$ for any finite extension $K$ of $k$, then \E is exact for $V$, cf. \cite[Theorem A]{Liang4}. If we only require the condition on $V_K$ hold for finite extensions that are linearly disjoint from a certain prefixed extension of $k$, the conclusion remains valid, cf. \cite[Theorem 2.1]{Liang5}.
\end{itemize}

The invariance of the exactness of \E and the weak approximation property for 0-cycles under the product with the projective line is well known. This can also be viewed as a baby case of our main result.
\begin{lem}\label{homotopylemma}
Let $V$ be a proper smooth geometrically integral variety defined over a number field $k$. Then
\begin{itemize}
\item[-] The variety $V$ satisfies weak approximation with \BMo~for 0-cycles of degree $\delta$ if and only if $\P^1\times V$ satisfies weak approximation with \BMo~for 0-cycles of degree $\delta$,
\item[-] the  sequence \E is exact for $V$ if and only if it is exact for $\P^1\times V$.
\end{itemize}
\end{lem}

\begin{proof}
It follows from the functoriality  of the Brauer group, of the Chow group of 0-cycles, and of the sequence $(\textup{E})$, applied to the natural projection $\P^1\times V\To V$ with an obvious section.
\end{proof}

\section{Main results and consequences}\label{mainresult}

Our main result is the following theorem, which roughly says that the property of  weak approximation with \BMo~for 0-cycles of degree $1$ is compatible with respect to products among certain classes of varieties including geometrically rationally connected varieties, K3 surfaces, and curves if appeared only once in a product.
\begin{thm}\label{mainthm}
Let $\delta$ be an integer.
Let $C$ be a smooth proper geometrically integral curve defined over a number field $k$. Assume that $C$ satisfies weak approximation with \BMo~for 0-cycles of degree $\delta$.
Let $Z$ be a finite product of proper smooth geometrically integral $k$-varieties, each of whose factor $V$ verifies the following conditions:
\begin{itemize}
\item[-] the groups $\Br(V_{\bar{k}})^{\Gamma_k}$ and $\Br_1(V)/\Br_0(V)$ are finite;
\item[-] for an arbitrary positive integer $d_0$, there exists a finite extension $k'$ of $k$ such that for any extension $K$ of $k$ of degree $d_0$ that is linearly disjoint from $k'$ over $k$, the base change morphism $\Br(V)/\Br_0(V)\To\Br(V_K)/\Br_0(V_K)$ is surjective;
\item[-] there exists a finite extension $k'$ of $k$ such that for any finite extension $K$ of $k$ that is linearly disjoint from $k'$ over $k$, the $K$-variety $V_K$ satisfies weak approximation with \BMo~ for 0-cycles of degree $1$.
\end{itemize}
Then $C\times Z$ satisfies weak approximation with \BMo~for 0-cycles of degree $\delta$.
\end{thm}

\begin{rem}
As remarked in \S \ref{BMpairing} and \S \ref{WABMnotion}, the hypotheses on the curve $C$ are verified when the Tate\textendash Shafarevich group of its jacobian $\textup{Jac}(C)$ is finite and if $C$ admits a global 0-cycle of degree $1$. When $\delta=1$, the existence of a global 0-cycle of degree $1$ is also a consequence of the finiteness of $\sha(\textup{Jac}(C))$ assuming the existence of a family of local 0-cycles of degree $1$ orthogonal to the Brauer group.
\end{rem}

\begin{rem}\label{exampleforBr}
When $V$ is a geometrically rationally connected variety, the finiteness condition on Brauer groups is verified. According to \cite[Proposition 3.1.1]{Liang4}, there exists a finite extension $k'$ of $k$, even independent of the integer $d_0$, such that for any extension $K$ of $k$ of degree $d_0$ linearly disjoint from $k'$ over $k$ the base change morphism $\Br(V)/\Br_0(V)\To\Br(V_K)/\Br_0(V_K)$ is  surjective.

When $V$ is a K3 surface, the finiteness of Brauer groups is proved by A. Skorobogatov and Yu. Zarhin \cite[Theorem 1.2]{SkZar08}. For an arbitrary positive integer $d_0$, the existence of an extension $k'$ in the hypothesis for the surjectivity of comparison of Brauer groups is proved by E. Ieronymou, we refer to the proof of \cite[Theorem 1.2]{Ieronymou19}.

When $V$ is a Kummer variety, the finiteness assumption on the Brauer groups is verified thanks to a result of A. Skorobogatov and Yu. Zarhin \cite[Corollary 2.8]{SkZarKummer}. But the comparison assumption on the Brauer groups under extension of the ground field may not hold. A modified comparison result is proved by F. Balestrieri and R. Newton in \cite{BaleNewton19}. A variant of our theorem still holds. We refer to \S \ref{RkKummer} after the proof of the theorem, where we explain how to adapt the statement to Kummer varieties.
\end{rem}

\begin{cor}\label{productBM}
Let $X$ and $Y$ be proper smooth geometrically integral varieties defined over a number field $k$. Suppose that $V=X$ and $Y$ verify the following conditions:
\begin{itemize}
\item[-] the groups $\Br(V_{\bar{k}})^{\Gamma_k}$ and $\Br_1(V)/\Br_0(V)$ are finite;
\item[-] for an arbitrary positive integer $d_0$, there exists a finite extension $k'$ of $k$ such that for any extension $K$ of $k$ of degree $d_0$ that is linearly disjoint from $k'$ over $k$, the base change morphism $\Br(V)/\Br_0(V)\To\Br(V_K)/\Br_0(V_K)$ is surjective;
\item[-] there exists a finite extension $k'$ of $k$ such that for any finite extension $K$ of $k$ that is linearly disjoint from $k'$ over $k$, the $K$-variety $V_K$ satisfies weak approximation with \BMo~ for 0-cycles of degree $1$.
\end{itemize}
(The first two conditions are verified for example by geometrically rationally connected varieties and K3 surfaces.)

Then  $X\times Y$  verifies weak approximation with \BMo~for 0-cycles of any degree.

In particular, if $X$ and $Y$ are geometrically rationally connected varieties such that \E is exact for $X_K$ and $Y_K$ for every finite extension $K$ of $k$, then \E is also exact for $X\times Y$.
\end{cor}

\begin{proof} The first statement follows from Theorem \ref{mainthm} applied to $C=\P^1$ and Lemma \ref{homotopylemma}.

When both varieties are geometrically rationally connected, according to \S \ref{WABMnotion}, the exactness of \E implies the property of weak approximation with \BMo~for 0-cycles of degree $1$, and the converse is also true if we allow finite extensions of the ground field.
\end{proof}

\begin{cor}\label{productWA}
Let $X$ and $Y$ be proper smooth geometrically integral varieties defined over a number field $k$. Suppose that after every finite extension of the ground field both varieties satisfy weak approximation for 0-cycles of degree $1$. Then the same property holds for $X\times Y$.
\end{cor}

\begin{proof}
Ignoring all arguments about Brauer groups (pretending that they are all $0$), the whole proof applies \textemdash~ we only need to consider the case with $C=\P^1$ whose Brauer group does not obstruct local to global properties.
\end{proof}

\begin{rem}\label{remWA0-cyc2}
This result is  predictable, but to the best knowledge of the author, probably it did not appear in the literature. One may also guess that the statement without extension of the ground field may possibly be valid as well. But this is unclear to the author, though its analogue for rational points is trivial. See also Remark \ref{remWA0-cyc1}.
\end{rem}

\begin{cor}
Let $X\To C$ be a dominant morphism between smooth proper geometrically integral varieties defined over a number field $k$.

Assume that $C$ is a curve such that the Tate\textendash Shafarevich group of its jacobian is finite.

Assume that the generic fiber is a product of geometrically rationally connected varieties over $k(C)$ and that each  fiber over a closed point $\theta$ in a certain non empty open subset (or a hilbertian subset, see \S \ref{Hil} for definition) of $C$ is a product of varieties that verify the exactness of \E after arbitrary finite extension of $k(\theta)$.

Then \E is also exact for $X$.
\end{cor}

\begin{proof}
By Lemma \ref{productBM}, the sequence \E is exact for fibers over all closed points in a certain non empty open subset (or a hilbertian subset) of $C$.
The statement is then a direct consequence of the main result of Y. Harpaz and O. Wittenberg \cite[Theorem 8.3]{HarWitt16}.
\end{proof}

\section{Proof of the main theorem}\label{proofsection}

\subsection{Preliminaries for the proofs}
Here we collect some known results in the literature that will contribute to our proof of the main theorem.
\subsubsection{Moving lemmas}
Recall that an effective 0-cycle is separable if it is a sum of distinct closed points.
Separable 0-cycles behave like rational points, to which it will be possible to apply the fibration method. We need the following moving lemmas to get effective separable 0-cycles from arbitrary 0-cycles.

The following moving lemma for 0-cycles is classic, one can find a proof in \cite[End of \S 3]{CT05},  for the quasi-projective case see also \cite{AK79}.

\begin{lem}\label{movinglemma0}
Let $X$ be a smooth integral variety defined over a field $k$ of characteristic $0$. Let $U$ be a non empty open subset of $X$. Then any 0-cycle on $X$ is rationally equivalent to a 0-cycle supported in $U$.
\end{lem}

On a smooth curve, sufficiently positive 0-cycles are rationally  equivalent to an effective 0-cycle. It follows from the Riemann\textendash Roch theorem.
\begin{lem}\label{movinglemma1}
Let $C$ be a projective smooth geometrically integral curve of genus $g$ defined over a field $k$ of characterisptic $0$. Let $z$ be a 0-cycle of degree $d>2g$ on $C$. Then $z$ is rationally equivalent to an effective  0-cycle $z'$.
\end{lem}

J.-L. Colliot-Th\'el\`ene generalised this lemma to a much stronger relative version. We include a sketch of his proof.

\begin{lem}[Relative moving lemma]\label{movinglemma2}
Let $f:X\To C$ be a morphism between projective smooth geometrically integral varieties defined over a field $k$ of characteristic $0$. Assume that $C$ is a curve. Let $z$ be a 0-cycle and $z^\eff$ be an effective 0-cycle on $X$.

Then there exists a positive integer $d_0$ such that for all $d>d_0$ the 0-cycle $z+dz^\eff$ is rationally equivalent on $X$ to an effective 0-cycle $\tau$ whose projection $f_*\tau$ to $C$ is separable.

\end{lem}
\begin{proof}[Sketch of proof.]
First of all one deals with the case where $X$ is a curve. One applies Riemann\textendash Roch theorem to get effective 0-cycles. Bertini's theorem ensures further the separation of the 0-cycle. For higher dimensional $X$, one can find a curve passing through the support of the concerned 0-cycle and hence one reduces to the case for curves. A detailed proof by J.-L. Colliot-Th\'el\`ene can be found in \cite[Lemmes 3.1 and 3.2]{CT99}.
\end{proof}

\subsubsection{Approximation for effective 0-cycles}
When $k$ is a number field, we consider its completions $k_v$ which will be $\mathbb{R}$, $\mathbb{C}$, or a finite extension of $\Q_p$. The set $\Sym_X^\Delta(k_v)$ of $k_v$-rational points of the symmetric product  endowed with a natural topology is identified with the set of effective 0-cycles of degree $\Delta$ on $X_{k_v}$. The following lemma by O. Wittenberg relates the topology on $\Sym_X^\Delta(k_v)$ with our notion of approximation for 0-cycles.
\begin{lem}[{\cite[Lemme 1.8]{Wittenberg}}]\label{WitLem}
Let $X$ be a projective smooth variety defined over $k_v$ and let $\Delta$ be a positive integer. For any positive integer $n$, the map $\Sym_X^\Delta(k_v)\To\CH_0(X_{k_v})/n$ which maps a degree $\Delta$ effective 0-cycle to its class  is locally constant.
\end{lem}

\subsubsection{Finiteness of Brauer groups for products}
In our proof, we will need the finiteness of the concerned Brauer groups. The following proposition, which allow us to deal with products, is proved by F. Balestrieri and R. Newton.

\begin{prop}[{\cite[Proposition 3.1]{BaleNewton19}}]\label{prodBrfinite}
Let $k$ be a number field and let $X$ and $Y$ be projective smooth geometrically integral $k$-varieties. If $\Br(V_{\bar{k}})^{\Gamma_k}$ and $\Br_1(V)/\Br_0(V)$ are finite for $V=X$ and $V=Y$, then they are also finite for $V=X\times Y$. Consequently $\Br(X\times Y)/\Br_0(X\times Y)$ is also finite.
\end{prop}

\subsubsection{Hilbert's irreducibility theorem}\label{Hil}
Let $V$ be a quasi-projective integral variety defined over a number field $k$. A hilbertian subset of $V$ was defined originally as the complement in $V(k)$ of a thin set in the sense of J.-P. Serre \cite[Chapter 9]{SerreMW}. It is regarded as a large subset. Hilbert's irreducibility theorem states that hilbertian subset of $\P^1$ is non empty. In \cite{Ekedahl}, T. Ekedahl proved an effective version of Hilbert's theorem taking care of approximation properties as well. One slightly extends this notion by considering also closed points of $V$ in order to deal with questions of 0-cycles. Recall the following definition of hilbertian subsets.

\begin{defn}
A subset $\Hil$ of closed points of $V$ is call a \emph{hilbertian subset} if there exist a non empty Zariski open subset $U$ of $V$, an integral $k$-variety $Z$, and a finite \'etale morphism $\rho:Z\To U$ such that $\Hil$ is the set of closed points $\theta$ of $U$ whose fiber $\rho^{-1}(\theta)$ is connected.
\end{defn}

The following proposition was proved by the author in \cite[Lemme 3.4]{Liang1}.
As its application to the proof of our main theorem is crucial, we include a proof for the convenience of the reader.

\begin{prop}\label{hilbertirred}
Let $C$ be a projective smooth geometrically integral curve of genus $g$ defined over a number field $k$. Let $S\subset\Omega_k$ be a finite set of places of $k$
Let $\Hil$ be a hilbertian subset of $C$. Suppose that  $y_\infty$ is a global effective 0-cycle of $C$ of degree $\Delta>2g$ and $z_v$ is a separable effective 0-cycle rationally equivalent to $y_\infty$ on $C_{k_v}$ for all $v\in S$.

Then there exists a closed point $\theta$ of $C$ such that
\begin{itemize}
\item[-] $\theta\in\Hil$,
\item[-]  for all $v\in S$, we have rational equivalence $\theta\sim y_\infty\sim z_v$ on $C_{k_v}$,
\item[-]  for all $v\in S$, identifying effective 0-cycles as rational points on the symmetric product of $C$, $\theta$ is arbitrarily close to $z_v$ in $\Sym^\Delta_{C}(k_v)$.
\end{itemize}
\end{prop}

\begin{proof}
By the moving lemma (Lemma \ref{movinglemma0}), we may assume that the support of $y_\infty$ is disjoint from the support of $z_v$ for all $v\in S$.

For each $v\in S$, write $z_v-y_\infty=\div_{C_{k_v}}(f_v)$ with a certain function $f_v\in k_v(C)^*/k^*_v$ up to a multiplication by a constant. Since $\deg(y_\infty)=\Delta>2g$, by Riemann\textendash Roch theorem, the set of sections $\Gamma(C,\mathcal{O}_C(y_\infty))$ is a vector space of dimension $d=\Delta+1-g>g+1$. Weak approximation applied to the projective space $\P^{d-1}$ give us a function $f\in k(C)^*/k^*$ such that $f$ is sufficiently close to $f_v$ for all $v\in S$. Whence, when we write $\div_C(f)=y_0-y_\infty$ the effective 0-cycle $y_0$ is sufficiently close to $z_v$ on $\Sym_C^\Delta(k_v)$ for all $v\in S$. Being sufficiently close to $z_v$ the 0-cycle $y_0$ is separable and its support is disjoint from the support of $y_\infty$.

The function $f$ defines a $k$-morphism $\alpha: C\To\P^1$ such that $\alpha^*(\infty)=y_\infty$ and $\alpha^*(0)=y_0$. It is \'etale at the support of the separable 0-cycle $y_0$.
If the hilbertian subset $\Hil$ of $C$ is defined by a finite \'etale morphism $Z\to U\subset C$ where $Z$ is an integral variety, then the composition $\beta:Z\to C\to \P^1$ defines a hilbertian subset $\Hil'$ of $\P^1$ in the following way: removing from $\P^1$ the finite set $\alpha(C\setminus U)$ of closed points and the branch locus of $\alpha$ we get an open subset $V$ of $\P^1$. Take $Z'$ to be the open subset $\beta^{-1}(V)$ of $Z$. Then the restriction of $\beta$ to $Z'$ being finite \'etale onto $V$ defines $\Hil'$. Moreover, we know that once $\theta'\in\Hil'$ then $\theta=\alpha^{-1}(\theta')$ must also be a closed point of $C$ belonging to $\Hil$. Now we apply an effective version by Ekedahl \cite{Ekedahl} of Hilbert's irreducibility theorem (see also \cite[Proposition 3.2.1]{Harari}), we obtain a $k$-rational point $\theta'\in\Hil'$ sufficiently close to $0\in\P^1(k_v)$ for all $v\in S$. Since $\alpha$ is \'etale at the support of $y_0$, the implicit function theorem implies that the closed point $\theta=\alpha^{-1}(\theta')\in\Hil$ is sufficiently close to $y_0$ and hence to $z_v$ on $\Sym_C^\Delta(k_v)$ for all $v\in S$. It is also clear that $\theta\sim y_\infty\sim z_v$ on $C_{k_v}$ for $v\in S$.
\end{proof}

\subsection{Proof of the main theorem}


The following lemma applies to a general fibration.

\begin{lem}\label{globalizationonbase}
Let $f:X\to Y$ be a dominant morphism between projective smooth geometrically integral varieties defined over a number field $k$.
Assume that $Y$ satisfies weak approximation with \BMo~for 0-cycles of degree $\delta$.

Fix a positive integer  $n$ and a finite set $S$ of places of $k$.

Then for any family $(x_v)_{v\in\Omega_k}$ of local 0-cycles of degree $\delta$ whose projection is orthogonal to $\Br(Y)$, there exist a global 0-cycle $y$ on $Y$ of degree $\delta$ and  a family $(x^0_v)_{v\in S}$ of local 0-cycles on $X$ of degree $0$ such that the projection of $x_v+nx^0_v$ to $Y$ is rationally equivalent to $y$  for all $v\in S$.
\end{lem}
\begin{proof}
Fix a closed point on the generic fiber, denote  by $m$ the degree of its residue field over $k(Y)$. Viewed as a 0-cycle of degree $m$ on the generic fiber, it extends to a family of 0-cycles $z$ of degree $m$ on the fibers $X_P$ over points $P$ contained in a certain open subset $U$ of $Y$. Then $f_*z=mP$.

Since the projection $(f_*x_v)_{v\in\Omega_k}$ is orthogonal to $\Br(Y)$, by assumption there exists a global 0-cycle $y$ on $Y$ of degree $\delta$ and a family $(y^0_v)_{v\in S}$ of local 0-cycles of degree $0$ such that for all $v\in S$ we have $y\sim f_*x_v+nmy^0_v$ on $Y_{k_v}$. We may assume that $y^0_v$ is supported in $U$ by the moving lemma (Lemma \ref{movinglemma0}). According to the argument in the previous paragraph, the 0-cycle $my^0_v$ must be the image by $f_*$ of a certain 0-cycle $x^0_v$ of degree $0$ on $X_{k_v}$. Then the projection of $x_v+nx^0_v$ is rationally equivalent to $y$ on $Y_{k_v}$ for all $v\in S$.
\end{proof}

The following lemma is observed by F. Balestrieri and R. Newton. It reduces the orthogonality to the product of the Brauer group to the orthogonality to the Brauer group of each factor. One should notice that the local rational points are considered over completions of $l$, which is a finite extension of $k$, while the concerned Brauer groups are of varieties over the base field $k$.
\begin{lem}[{\cite[Lemma 3.4]{BaleNewton19}}]\label{BNlem}
Let $X$ and $Y$ be projective smooth geometrically integral varieties defined over a number field $k$. Let $l$ be a finite extension of $k$. Then in $\prod_{w\in\Omega_l}(X\times Y)(l_w)$ we have the following inclusion
$$[\prod_{w\in\Omega_l}(X\times Y)(l_w)]^{\Br(X\times Y)}\subset [\prod_{w\in\Omega_l}X(l_w)]^{\Br(X)}\times [\prod_{w\in\Omega_l}Y(l_w)]^{\Br(Y)}$$
where by abuse of notation $\Br(V)$ stands for the image of $\Br(V)\to\Br(V_l)$ for a $k$-variety $V$.
\end{lem}
\begin{proof}
It follows from the functoriality of the \BM~pairing applied to natural projections $X\times Y\to X$ and $X\times Y\to Y$.
\end{proof}

Recall some basic operators on (classes of) algebraic cycles (mostly 0-cycles), cf. \cite[Chapters 1 and 8]{Fulton}. Let $f:X\To Y$ be a proper morphism between varieties. For any closed point $P$ on $X$ with image $Q=f(P)$, the homomorphism $f_*:\ZZ(X)\To \ZZ(Y)$ defined by $f_*(P)=[k(P):k(Q)]Q$ induces a homomorphism $f_*:\CH_0(X)\To\CH_0(Y)$. Let $f:X\To Y$ be a flat morphism between varieties of relative dimension $n$. Then $\ZZ(Y)\To\textup{Z}_n(X),P\mapsto f^{-1}(P)$ induces a homomorphism $f^*:\CH_0(Y)\To \CH_n(X)$. When $f^*$ arises via a base extension of the ground field over which the concerned variety is defined, usually for the extension $k_v|k$ in our proof, we will omit $f^*$ in such a particular case to simplify the notation. For a smooth $d$-dimensional $k$-variety $V$, there exists an intersection product on classes of algebraic cycles $\CH_m(V)\times\CH_n(V)\overset{\cap}\To\CH_{m+n-d}(V)$, where $\CH_k(V)$ denotes the Chow group of $k$-dimensional cycles.

\begin{lem}\label{keylemma}
Let $X$ and $Y$ be proper smooth geometrically integral varieties defined over a number field $k$. Let $S\subset\Omega_k$ be a finite subset of places of $k$. Let $x\in\CH_0(X)$ and $y\in\CH_0(Y)$ be classes of global 0-cycles together with a family of classes of local 0-cycles of degree $1$ $(z_v)_{v\in S}\in\prod_{v\in S}\CH_0((X\times Y)_{k_v})$    such that for a given positive integer $n$:
\begin{itemize}
\item[-] the classes $x$ and ${p_{1}}_*(z_v)$ coincide in $\CH_0(X_{k_v})/n$ for all $v\in S$,
\item[-] the classes $y$ and ${p_{2}}_*(z_v)$ coincide in $\CH_0(Y_{k_v})/n$ for all $v\in S$,
\end{itemize}
where $p_1:X\times Y\to X$ and $p_2:X\times Y\to Y$ are natural projections.

Suppose in addition that
the class of 0-cycle $z_v$ can be represented by a $k_v$-rational point on $X\times Y$ for all $v\in S$.

Then $p_1^*(x)\cap p_2^*(y)$ and $z_v$ coincide in $\CH_0((X\times Y)_{k_v})/n$ for all $v\in S$.
\end{lem}

\begin{proof}
By assumption, let $z_v$ be represented by a $k_v$-rational point $(x_v,y_v)\in X(k_v)\times Y(k_v)$. Then ${p_1}_*(z_v)$ is represented by $x_v$ and ${p_2}_*(z_v)$ is represented by $y_v$. Under the intersection product
$$\CH_{\dim Y}((X\times Y)_{k_v})/n\times \CH_{\dim X}((X\times Y)_{k_v})/n\xrightarrow{\cap} \CH_0((X\times Y)_{k_v})/n,$$
the classes $p_1^*(x)\cap p_2^*(y)$ and $p_1^*{p_1}_*(z_v)\cap p_2^*{p_2}_*(z_v)$ coincide in $\CH_0((X\times Y)_{k_v})/n$. The latter being represented by $p_1^*(x_v)\cap p_2^*(y_v)=(x_v,y_v)$ is nothing but $z_v$.
\end{proof}

\begin{rem}\label{remWA0-cyc1}
This is a statement on weak approximation for 0-cycles on a product of varieties. Its analogue for rational points follows trivially from the definition of the product topology. However, for 0-cycles, our proof requires the additional hypothesis that $z_v$ is rationally equivalent to a local rational point to ensure that
$p_1^*{p_1}_*(z_v)\cap p_2^*{p_2}_*(z_v)=z_v\in\CH_0((X\times Y)_{k_v})$.
Without such a hypothesis, there is no reason to expect the equality to hold. The hypothesis seems to be very restrictive, but finally the lemma turns out to be useful for our goal. See also Remark \ref{remWA0-cyc2}.
\end{rem}

Now we are ready to present the proof of Theorem \ref{mainthm} with full details.

\begin{proof}[\textbf{Proof of Theorem \ref{mainthm}}]
We will deal with the case where $Z$ has exactly two factors. The proof also applies to the case where $Z$ has only one factor.
The general case with more factors follows by induction on the number of factors applied to Corollary \ref{productBM} which is a consequence of the special case of this theorem with exactly two factors.

If $V$ is a proper smooth variety, then by Chow's lemma and Hironaka's theorem, there exists a birational morphism $V'\to V$ with $V'$ projective and smooth. Birational morphisms induce isomorphisms between Brauer groups and isomorphisms between Chow groups of 0-cycles. Therefore we may suppose that $Z$ is a product of two projective varieties.

Let $(z_v)_{v\in\Omega_k}$ be a family of local 0-cycles of degree $\delta$ on $C\times Z$ orthogonal to $\Br(C\times Z)$  under the \BM~pairing. Let $S\subset\Omega_k$ be a finite set of places of $k$. Let $n$ be a positive integer. Our goal is to find a global 0-cycle $z=z_{n,S}$ of degree $\delta$ such that it has the same image as $z_v$ in $\CH_0(C_{k_v}\times Z_{k_v})/n$ for all $v\in S$.

We view $C\times Z$ as a trivial fibration $f:C\times Z\To C$. Our strategy is to apply the fibration method, it divides roughly into three steps:
\begin{itemize}
\item[-]    arrange the local 0-cycles to a good position with respect to the fibration;
\item[-]   approximate the image of the local 0-cycles under $f_*$ by a single closed point $\theta$ of $C$;
\item[-]   apply the hypothesis on weak approximation to the fiber $Z_\theta=f^{-1}(\theta)$.
\end{itemize}

First of all, without loss of generality we may make some more assumptions on $n$ and $S$.
\begin{itemize}
\item[-] Since $\Br(Z)/\Br_0(Z)$ is finite by Lemma \ref{prodBrfinite}, by replacing $n$ by the product of $n$ and the order of this group, we may assume that $n$ annihilates elements of $\Br(Z)/\Br_0(Z)$.
\item[-] We may assume that $S$ contains all archimedean places.
\item[-] We may assume that $S$ is sufficiently large such that the evaluations of the images in $\Br(C\times Z)$ of (a finite complete set of) representatives of $\Br(Z)/\Br_0(Z)$ at any local rational points of $C\times Z$ for places outside $S$ are identically $0$. This follows from a standard good reduction argument and the fact that the Brauer group of a complete discrete valuation ring is trivial.
\item[-] According to Lang\textendash Weil estimate and Hensel's lifting lemma, we may assume that the finite set $S$ is sufficiently large such that for any $v\notin S$ the smooth geometrically integral variety $Z$ possesses $k_v$-rational points.
\end{itemize}

By the functoriality of the \BM~pairing, the family of local degree $\delta$ 0-cycles $(f_*z_v)_{v\in\Omega_k}$ is orthogonal to $\Br(C)$.p
We apply Lemma \ref{globalizationonbase} which says that up to an $n$-divisible part the projection of the local 0-cycles to $C$ are  rationally equivalent to a global 0-cycle of degree $\delta$. More precisely, there exist a global 0-cycle $y$ on $C$ of degree $\delta$ and  a family $(z^0_v)_{v\in S}$ of local 0-cycles on $C\times Z$ of degree $0$ such that
$f_*z'_v\sim y$ in $\ZZ(C_{k_v})$ for $v\in S$ where $z'_v=z_v+nz^0_v\in \ZZ(C\times Z)$. For $v\notin S$, set $z'_v=z_v$.
Since $n$ annihilates the $\Br(Z)/\Br_0(Z)$, the family $(z'_v)_{v\in\Omega_k}$ is also orthogonal to the image of $\Br(Z)\to\Br(C\times Z)$  under the \BM~pairing.

Fix a global effective 0-cycle $z^\eff$ (for example a closed point) on $C\times Z$. For each $v\in S$, we apply the relative moving lemma (Lemma \ref{movinglemma2}) to $k_v$-fibration $f:C_{k_v}\times Z_{k_v}\to C_{k_v}$ to obtain a positive integer $d_v$ such that for $d>d_v$ we have $z'_v+dz^\eff\sim\tau_v$ where $\tau_v$ is a certain effective 0-cycle  $C_{k_v}\times Z_{k_v}$ whose projection $f_*\tau_v\in\ZZ(C_{k_v})$ is separable. A fortiori, the 0-cycle $\tau_v$ itself is separable. We fix an integer $d>\max\{d_v;v\in S\}$ sufficiently large such that the moving lemma (Lemma \ref{movinglemma1})  also applies to  $C$: $y+df_*z^\eff\sim y_\infty$ where $y_\infty$ is a certain effective 0-cycle on $C$ of very large degree. It is clear that $y_\infty\sim f_*\tau_v$ on $C_{k_v}$ for all $v\in S$. For $v\notin S$, we set $\tau_v=z'_v+dz^\eff$ which is not separable and which may not be effective.
Since the difference of $(\tau_v)_{v\in\Omega_k}$ and $(z'_v)_{v\in\Omega_k}$ is rationally equivalent to a global 0-cycle, the family $(\tau_v)_{v\in\Omega_k}$ is  orthogonal to the image of $\Br(Z)\to\Br(C\times Z)$  under the \BM~pairing.
Denote by $\Delta$ the degree of the 0-cycles $y_\infty$ and  $\tau_v$.

For $d_0=\Delta$ take a finite extension $k'$ of $k$ satisfying the two conditions concerning $k'$ in the statement of the theorem. The natural projection $C_{k'}\to C$ defines a hilbertian subset $\Hil$ of $C$. Then a closed point $\theta$ of $C$ belongs to $\Hil$ if and only if its residue field $k(\theta)$ and $k'$ are linearly disjoint over $k$. The two conditions will apply to $K=k(\theta)$.

Now we are ready to apply the extended version of Hilbert's irreducibility theorem (Proposition \ref{hilbertirred}) to find a closed point $\theta$ on $C$ verifying the following conditions:
\begin{itemize}
\item[-] $\theta\in\Hil$,
\item[-] for all $v\in S$,  a rational equivalence $\theta\sim y_\infty\sim f_*\tau_v$ on $C$,
\item[-] for all $v\in S$, identifying effective 0-cycles as rational points on the symmetric product of $C$, $\theta$ is arbitrarily close to $f_*\tau_v$ in $\Sym^\Delta_{C}(k_v)$,
\end{itemize}

For $v\in S$, since $f:C\times Z\to C$ is a smooth morphism and $f_*\tau_v$ is separable, the induced morphism $f_*:\Sym^\Delta_{C\times Z}(k_v)\to\Sym^\Delta_{C}(k_v)$ is smooth at $\tau_v$. It follows from the implicit function theorem that for each $v\in S$ there exists $\tau'_v\in \Sym^\Delta_{C\times Z}(k_v)$ arbitrarily close to $\tau_v$ and such that $f_*\tau'_v=\theta$. Therefore $\tau'_v$ is an effective local 0-cycle of degree $\Delta$ on $C_{k_v}\times Z_{k_v}$ sitting exactly on the fiber $f^{-1}(\theta)\simeq Z_{k(\theta)}$. The pull-back $\Spec(k(\theta))\otimes_{C}C_{k_v}=\Spec(k(\theta)\otimes_kk_v)$ of the closed point $\theta$ of $C$ to $C_{k_v}$ is an effective separable 0-cycle written as $\sum_{w|v}\theta_w$ where $\theta_w$ is a closed point of $C_{k_v}$ of residue field $k(\theta)_w$. The equality $f_*\tau'_v=\theta\in\ZZ(C_{k_v})$ means that the separable 0-cycle $\tau'_v$ is written as $\tau'_v=\sum_{w|v}\tau'_w$ where $\tau'_v$ is a $k(\theta)_w$-rational point of $C_{k_v}\times Z_{k_v}$ sitting exactly on the fiber $f^{-1}(\theta_w)$. Therefore we get a local rational point on the fiber $f^{-1}(\theta)\simeq Z_{k(\theta)}$ for each place $w$ of $k(\theta)$ lying over a place $v\in S$.

For any place $w$ of $k(\theta)$ lying over a place $v\notin S$, by the choice of $S$, there exist $k_v$-rational points on $Z$ and hence $k(\theta)_w$-rational points on $f^{-1}(\theta)\simeq Z_{k(\theta)}$. We fix such a $k(\theta)$-rational point and denote it by $\tau'_w$ and put $\tau'_v=\sum_{w|v}\tau'_w\in\ZZ(C_{k_v}\times Z_{k_v})$. We get a family $(\tau'_w)_{w\in\Omega_{k(\theta)}}$ of local rational points on $f^{-1}\simeq Z_{k(\theta)}$ which can also be viewed as a family $(\tau'_v)_{v\in\Omega_k}$ of local 0-cycles of degree $\Delta$ on $C\times Z$.

We are verifying the orthogonality of this family with the image of $\Br(Z)\to\Br(C\times Z)$  under the \BM~pairing. For $w$ lying over $v\notin S$, by the choice of $S$ the 0-cycles $\tau'_w$ and $\tau_v$ contribute nothing to the \BM~pairing. For $w$ lying over $v\in S$, since $\tau'_v$ is sufficiently close to $\tau_v\in\Sym^\Delta_{C\times Z}(k_v)$, they have the same image in $\CH_0(C_{k_v}\times Z_{k_v})/n$ by Lemma \ref{WitLem}. As $n$ annihilates $\Br(Z)/\Br_0(Z)$, the evaluations of the images in $\Br(C\times Z)$ of representatives of $\Br(Z)/\Br_0(Z)$ at $\tau'_v$ and at $\tau_v$ are equal. Summing up, $(\tau'_v)_{v\in\Omega_k}$ (or $(\tau'_w)_{w\in\Omega_{k(\theta)}}$) are orthogonal to the image of $\Br(Z)\to\Br(C\times Z)$ since it is the case for $(\tau_v)_{v\in\Omega_k}$.

The commutative diagram on the left induces the commutative diagram on the right for Brauer groups.
$$\xymatrix{
Z& f^{-1}(\theta)=\theta\times Z\ar[l]\ar[dl]&&\Br(Z)\ar[d]\ar[r]&\Br(f^{-1}(\theta))=\Br(Z_{k(\theta)})\\
C\times Z \ar[u]&  &&\Br(C\times Z)\ar[ru]&
}$$
View $(\tau'_w)_{w\in\Omega_{k(\theta)}}$ as a family of local rational points on the fiber $f^{-1}(\theta)$ identified with the $k(\theta)$-variety $Z_{k(\theta)}$. It is orthogonal to the image of $\Br(Z)\to\Br(Z_{k(\theta)})$ by functoriality of the \BM~pairing.

As $Z$ is the product of two varieties $X$ and $Y$, then we write the $k(\theta)_w$-rational point $\tau'_w$ on $Z_{k(\theta)}=X_{k(\theta)}\times Y_{k(\theta)}$ as $\tau'_w=(\alpha_w,\beta_w)$. Lemma \ref{BNlem} says that $(\alpha_w)_{w\in\Omega_{k(\theta)}}$ is orthogonal to the image of $\Br(X)\to\Br(X_{k(\theta)})$ and $(\beta_w)_{w\in\Omega_{k(\theta)}}$ is orthogonal to the image of $\Br(Y)\to\Br(Y_{k(\theta)})$. As $\theta\in\Hil$, the residue field $k(\theta)$ is linearly disjoint from $k'$ over $k$, the last two morphisms between Brauer groups are actually surjections by assumption. Weak approximation with \BMo~for 0-cycles of degree $1$ on $X_{k(\theta)}$ (respectively $Y_{k(\theta)}$) implies the existence of a global 0-cycle $\alpha\in\ZZ(X_{k(\theta)})$ (respectively $\beta\in\ZZ(Y_{k(\theta)})$) of degree $1$ such that its classes coincide with the class of $\alpha_w$ in $\CH_0(X_{k(\theta)_w})/n$ (respectively $\CH_0(Y_{k(\theta)_w})/n$) for all $w$ lying over $v\in S$.

Now we apply Lemma \ref{keylemma} to obtain a global 0-cycle $\tau'=\tau'_{n,S}$ of degree $1$ on the $k(\theta)$-variety $Z_{k(\theta)}=X_{k(\theta)}\times Y_{k(\theta)}$ whose class and the class of $\tau'_w$ have the same image in $\CH_0(Z_{k(\theta)_w})/n$ for all $w$ lying over $v\in S$. Being denoted still by $\tau'$, it can be regarded as a global 0-cycle of degree $\Delta$ on $C\times Z$. We have seen that the images of the classes of $\tau'_v$ and $\tau_v$ are the same in $\CH_0(C_{k_v}\times Z_{k_v})$ for all $v\in S$. The commutative diagram
$$\xymatrix{
\CH_0(Z_{k(\theta)})/n\ar[r]\ar[d]&\prod_{w|v}\CH_0(Z_{k(\theta)_w})/n\ar[d]^{\sum_{w|v}} \\
\CH_0(C\times Z)/n\ar[r]&\CH_0(C_{k_v}\times Z_{k_v})/n
}$$
and the formula $\tau'_v=\sum_{w|v}\tau'_w$ implies that the images of the classes of $\tau'$ and $\tau_v$ are the same in $\CH_0(C_{k_v}\times Z_{k_v})/n$ for all $v\in S$. Define $z=\tau'-dz^\eff$, it is a global 0-cycle of degree $\delta$ on $C\times Z$. Remember that $\tau_v\sim z_v+nz^0_v+dz^\eff$ on $C_{k_v}\times Z_{k_v}$. We finally check that the classes of $z$ and $z_v$ have the same image in $\CH_0(C_{k_v}\times Z_{k_v})/n$ for all $v\in S$, which completes the proof.
\end{proof}

\subsection{A remark for Kummer varieties}\label{RkKummer}
In Remark \ref{exampleforBr}, we have explained that when $V$ is a geometrically rationally connected variety or a K3 surface then the finiteness condition on the Brauer groups and the comparison of  Brauer groups under base extensions are both verified, and hence Theorem \ref{mainthm} applies. When $V$ is a Kummer variety arising from a 2-covering of an abelian variety of dimension $\geq2$ (see \cite[Definition 2.1]{SkZarKummer} for a precise definition), as previously remarked the finiteness condition on Brauer groups of Theorem \ref{mainthm} is verified. In general, it is not clear whether we have a comparison under base extensions of the ground field for the whole Brauer group of a Kummer variety. The work of  F. Balestrieri and R. Newton shows that for a  positive integer $c$, the $c$-primary part is unchanged for certain base extensions:
\begin{lem}[{\cite[Lemma 7.1(3)]{BaleNewton19}}]\label{KummerBr}
Let $V$ be a Kummer variety defined over a number field $k$. Then there exists a finite extension $k'$ of $k$ such that for any positive integer $c$ and all finite extensions $K$ of degree coprime to $c$ and linearly disjoint from $k'$ over $k$, the base change morphism
$$(\Br(V)/\Br_0(V))\{c\}\To(\Br(V_K)/\Br_0(V_K))\{c\}$$
is an isomorphism.
\end{lem}
We may adapt this lemma to the proof of Theorem \ref{mainthm} in order to allow Kummer varieties to appear in the factors of $Z$, but the precise statement should be slightly modified as follows (if at least one factor $V$ of $Z$ is a Kummer variety).

\noindent Modified assumptions:
\begin{itemize}
\item[-] For each Kummer variety factor $V$ of $Z$, for any finite extension $K$ of $k$ linearly disjoint from $k'$, the variety $V_K$ satisfies weak approximation with $c$-primary \BMo~for 0-cycles of degree $1$.
\item[-] Additionally, assume that $\delta$ is coprime to $c$.
\end{itemize}
\noindent Modified conclusion:
\begin{itemize}
\item[-] 0-cycles of degree $\delta$ on $C\times Z$ satisfy weak approximation with $c'$-primary \BMo, where $c'$ is the least common multiple of $c$ and $|\Br(V')/\Br_0(V')|$ for all other non-Kummer-variety factors $V'$ of $Z$.
\end{itemize}
\noindent Modifications of the proof:
\begin{itemize}
\item[-] Starting from $(z_v)_{v\in\Omega_k}$ orthogonal to $c'$-primary part of $\Br(C\times Z)/\Br_0(C\times Z)$, we only need to consider a complete set of representatives of the image of $(\Br(Z)/\Br_0(Z))\{c'\}$. We get $(\tau'_w)_{w\in\Omega_{k(\theta)}}=((\alpha_w,\beta_w))_{w\in\Omega_{k(\theta)}}$ whose factor is orthogonal to $(\Br(V)/\Br_0(V))\{c'\}$ respectively for $V=X$ or $V=Y$. If $V$ is not a Kummer variety, the last group equals to the whole group $\Br(V)/\Br_0(V)$ by  definition of $c'$, and one can continue with the usual proof. If $V$ is a Kummer variety, the last group contains the $c$-primary part, and one can apply Lemma \ref{KummerBr} with the modified assumptions. Indeed, in the proof one may  further choose the sufficiently large integer $d$ to be divisible by $c$, then $[k(\theta):k]=\Delta\equiv\delta(\textup{mod}~c)$ is coprime to $c$.
\end{itemize}


\small


\bibliographystyle{alpha}
\bibliography{mybib1}
\end{document}